\documentclass[12pt]{amsart}

\usepackage{amsmath, amsfonts, amsbsy, amsthm, amscd, graphicx}
\usepackage{amssymb, latexsym, color}
\usepackage{slashed}
\usepackage{url}

\theoremstyle{plain}
\newtheorem{Theorem}{Theorem}[section]
\newtheorem{Proposition}[Theorem]{Proposition}
\newtheorem{Lemma}[Theorem]{Lemma}
\newtheorem{Corollary}[Theorem]{Corollary}

\newtheorem{Problem}[Theorem]{Problem}

\theoremstyle{definition}
\newtheorem{Definition}{Definition}

\theoremstyle{remark}
\newtheorem{Remark}{{\bf Remark}}
\newtheorem{Example}{Example}

\newcommand{\C}{{\mathbb{C}}}
\newcommand{\R}{{\mathbb{R}}}
\newcommand{\Z}{{\mathbb{Z}}}
\newcommand{\N}{{\mathbb{N}}}
\newcommand{\diam}{\mathrm{diam}}
\newcommand{\vol}{\mathrm{Vol}}
\newcommand{\Var}{\mathrm{Var}}
\newcommand{\var}{\mathrm{var}}

\newcommand{\infldim}{\mathrm{infldim}}

\newcommand{\Area}{\mathrm{Area}}
\newcommand{\Vol}{\mathrm{Vol}}

\begin{document}

\title{First-eigenvalue maximization and inflation of maps}

\author{Shin Nayatani}
\address{Graduate School of Mathematics, Nagoya University, Chikusa-ku, Nagoya 464-8602, Japan}
\email{nayatani@math.nagoya-u.ac.jp}


\maketitle

\begin{abstract}
Given a compact manifold equipped with a volume element and a Riemannian metric, we formulate
and study a dual pair of optimization problems: one concerning $C^\infty$-maps from the manifold
into the Hilbert space $l^2$ and the other concerning the smallest positive eigenvalue of the
Bakry-\'Emery Laplacian.
We present examples of manifolds for which these problems can be solved explicitly.
We also prove a Nadirashvili-type theorem.
\end{abstract}

\section*{Introduction}
In this paper, given a compact manifold on which a volume element and a Riemannian metric are specified, we formulate and study an optimization problem concerning $C^\infty$-maps from the manifold into $l^2$ and a maximization problem, dual to the former one in an appropriate sense, for the smallest positive eigenvalue of a certain type of Laplacian.
Here, $l^2$ denotes the Hilbert space consisting of square-summable sequences of real numbers.
Throughout this paper, we assume that a manifold is connected and has no boundary.

The celebrated Nash isometric embedding theorem \cite{Nash} states that any compact Riemannian manifold $(M, h)$ of dimension $n$ can be isometrically embedded into a Euclidean space $\R^N$, where $N=n(3n+11)/2$.
If $N$ is allowed to be larger, the image of the Nash embedding $\varphi\colon M\to \R^N$ can be put in an arbitrarily small ball.
On the other hand, the image cannot be too large.
In fact, we have the obvious bound $\diam \varphi(M)\leq \diam (M, h)$, where the left-hand side is the extrinsic diameter of $\varphi(M)$ in $\R^N$ and the right-hand side is the intrinsic diameter of the Riemannian manifold $(M, h)$.
One is therefore led to the problem of searching for an isometric embedding with the largest image in an appropriate sense.
Such an embedding, if exists, may be regarded as a standard isometric embedding.

We propose one possible definition of a map with the largest image, by considering an analogue for a manifold of the problem formulation for a graph introduced by G\"oring-Helmberg-Wappler \cite{GoeringHelmbergWappler1, GoeringHelmbergWappler2}.
Thus, when a volume element $d\mu$ and a Riemannian metric $h$ are given, we consider the optimization problem (Problem \ref{problem-emb}) of maximizing the variance $\var(\varphi) = \int_M\|\varphi\|^2\, d\mu_1$, where $d\mu_1=d\mu/\Vol(d\mu)$, over all $C^\infty$-maps $\varphi\colon M\to l^2$ satisfying the constraint
\begin{equation}\label{emb-constraint}
\varphi^* h_{l^2}\leq h,\quad \int_M\varphi\, d\mu_1=0,
\end{equation}
where $h_{l^2}$ is the standard Riemannian metric of $l^2$.
We call a map attaining the supremum of variance, $\Var(d\mu,h)$, an {\em inflated map} if exists.
A map satisfying the former condition in \eqref{emb-constraint} is named a {\em short} map by Nash.
Note that shortness is equivalent to being $1$-Lipschitz within the category of smooth maps.
On the other hand, the functional $\var(\varphi)$ measures globally how the map expands in the space $l^2$.
Therefore, an inflated map is a globally most expanding map among all locally shrinking (non-expanding) maps into $l^2$.
In view of examples in Section 3, an inflated map, if exists, is not necessarily an isometric immersion.
We, however, believe that an inflated map is an isometric immersion for a sufficiently large class of pairs $(d\mu, h)$.

The study of the maximization of the smallest positive eigenvalue of the Laplacian over all Riemannian metrics (with volume normalized) on a compact manifold began with the seminal work of Hersch \cite{Hersch}.
He proved that for any Riemannian metric on the two-sphere $S^2$, the inequality $\lambda_1(g)\cdot \Area(g)\leq 8\pi$ holds, where $\lambda_1(g)$ denotes the smallest positive eigenvalue, called the first eigenvalue hereafter, of the Laplacian with respect to $g$ and $\Area(g)$ denotes the area of $g$.
Subsequently, Berger \cite{Berger} proposed the problem whether for any compact manifold $M$, the scale-invariant functional $g\mapsto \lambda_1(g)\cdot \Vol(g)^{2/n}$, which we call the {\em Berger functional}, defined on all Riemannian metrics is bounded from above or not.
Here, $\Vol(g)$ denotes the volume of $g$ and $n$ is the dimension of $M$.
Urakawa \cite{Urakawa} then gave the first negative answer to this problem by computing the first eigenvalue of the so-called Berger metrics on $S^3$.
These metrics are obtained by rescaling the standard metric by a fixed ratio in the direction of the fibers of the Hopf map $S^3\to S^2$.
Urakawa showed that the Berger functional diverges to infinity as this ratio tends to infinity.
On the other hand, Yang-Yau \cite{YangYau} proved that the inequality $\lambda_1(g)\cdot \Area(g)\leq 8\pi(\gamma+1)$ holds for any Riemannian metric $g$ on the compact surface of genus $\gamma$.
(El Soufi-Ilias \cite{ElSoufiIlias} noted that the constant $\gamma+1$ in this inequality can be improved to $\bigl[\frac{\gamma+3}{2}\bigr]$.)
Therefore, the Berger problem was affirmatively solved for compact surfaces.
After the work of Yang-Yau, the problem of finding the supremum of the Berger functional on a compact surface of fixed genus has been actively studied.
See \cite{Nadirashvili, JLNNP, NayataniShoda, Ros, KarpukhinVinokurov} for developments in this direction.
For general genera, the existence of a metric (possibly with conical singularities) that attains the supremum of the Berger functional was finally established by Matthiesen-Siffert \cite{MatthiesenSiffert}, following the crutial contribution by Petrides \cite{Petrides}.
Meanwhile, Nadirashvili \cite{Nadirashvili} proved that if a Riemannian metric $g$ attains the supremum of the Berger functional on a compact surface, then there exists a map consisting of first eigenfunctions of the Laplacian $-\Delta_g$, which gives an isometric minimal immersion into a sphere.
This is a beautiful theorem that bridges the eigenvalue maximization problem and minimal surface theory, which are seemingly unrelated.

We propose a new first-eigenvalue maximization problem.
Again, we consider the setting in which a volume element $d\mu$ and a Riemannian metric $h$ are given on a compact manifold $M$.
Then we consider the problem (Problem \ref{problem-spec}) of maximizing the first eigenvalue
$\lambda_1(d\mu,g)$ of the Bakry-\'Emery Laplacian
$$
-\Delta_{(d\mu,g)} = -\Delta_g + g(df, d\bullet),\quad d\mu_g = e^f\, d\mu,
$$
where $d\mu_g$ is the volume element of $g$, over all Riemannian metrics $g$ satisfying the constraint
\begin{equation}\label{metric-normalization-intro}
\int_M (g, h)\, d\mu_1 =1,\quad (g,h) = g^{ij}h_{ij},
\end{equation}
and denote the supremum of $\lambda_1(d\mu,g)$ by $\Lambda(d\mu,h)$.
This problem is, in fact, equivalent to the Lagrangian dual problem of the previously described optimization problem (the primal problem) concerning maps.
As a consequence of duality, we obtain the weak duality inequality $\Var(d\mu,h)\leq \frac{1}{\Lambda(d\mu,h)}$
between the optimal values of the primal and dual problems.

While the idea of Lagrangian duality is useful for identifying the dual problem, the weak duality inequality itself follows immediately from the relation $\var(\varphi)\leq \frac{1}{\lambda_1(d\mu,g)}$ between the functionals, where $\varphi$ and $g$ satisfy the constraints \eqref{emb-constraint} and \eqref{metric-normalization-intro}, respectively.
Though the proof of this inequality is easy, it should be noted that the former constraint in \eqref{emb-constraint} is crutial for relating two problems involving different variables.
The equality condition is that $\varphi$ and $g$ satisfy
\begin{equation}\label{condition-equality-intro}
-\Delta_{(d\mu,g)} \varphi = \lambda_1(d\mu,g) \varphi,\quad \varphi^* h_{l^2} = h.
\end{equation}
When we specified $d\mu=d\mu_h$ initially, the necessary and sufficient condition for the equality to hold with the choice of $g=h$ reduces to the Takahashi condition
$$
-\Delta_h \varphi = \lambda_1(h) \varphi,\quad \varphi^* h_{l^2} = h.
$$
According to the Takahashi theorem \cite{Takahashi}, this condition is equivalent to the statement that $\varphi$ is an isometric minimal immersion into a sphere of some radius
by first eigenfunctions of the Laplacian $-\Delta_h$.
In particular, it follows that such an immersion is an inflated map.  
Many examples of compact Riemannian manifolds $(M,h)$ which admit an isometric minimal immersion into a sphere by first eigenfunctions are known, and for those examples, the primal and dual problems with respect to the choice $d\mu=d\mu_h$ are solved in a trivial manner.

The simplest nontrivial examples are flat tori.
For the two-dimensional torus $T^2$, we can solve the primal and dual problems associated with the pair $(d\mu_h, h)$, where $h$ is any flat metric on $T^2$ and $d\mu_h$ is its volume element.
In fact, the flat metric $h_{\mathrm{EL}}$ corresponding to the equilateral lattice $\Z(1,0)\oplus \Z(1/2, \sqrt{3}/2)$ gives a {\em common} solution to all these first-eigenvalue maximization problems, and maps consisting of first eigenfunctions of the Laplacian with respect to $h_{\mathrm{EL}}$ give inflated isometric embeddings into $\R^6$.

The Berger spheres mentioned above provide interesting three-dimensional examples.
Let $h_t$ denote the Berger metric obtained by rescaling the standard metric of $S^3$ by a factor of $t^2$ in the direction of the Hopf fibers.
In the already mentioned work \cite{Urakawa}, Urakawa observed that the multiplicity of the first eigenvalue of $h_t$ becomes $7$ at $t=1/\sqrt{6}$, where two different eigenspaces meet.
The solutions of the primal and dual problems associated with the pair $(d\mu_{h_t}, h_t)$, consisting of the Berger metric $h_t$ and its volume element, exhibit rather different behavior depending on whether $t\leq 1$ or $t\geq 1$.
For $t\leq 1$, the Berger metric $h_{1/\sqrt{6}}$ gives a {\em common} solution to all the first-eigenvalue maximization problems with $t\leq 1$, and maps consisting of first eigenfunctions of the Laplacian with respect to $h_{1/\sqrt{6}}$ give inflated isometric embeddings into $\R^7$.
On the other hand, for $t\geq 1$, at least among left-invariant metrics on $S^3=\mathrm{SU}(2)$, optimal solutions to the first-eigenvalue maximization problems cannot be found as Riemannian metrics.
In fact, the Carnot-Carath\'eodory metric $h_\infty=\lim_{t\to \infty} h_t$ gives a common solution to all these problems with $t\geq 1$.
We also obtain inflated maps as maps consisting of first eigenfunctions of $h_\infty$, which is nothing but the standard inclusion map of $S^3$ into $\R^4$, and is therefore not isometric with respect to $h_t$ (unless $t=1$).
Note that this does not contradict the latter equation of \eqref{condition-equality-intro} since $h_\infty$ is not a Riemannian metric.

Not limited to the Berger metrics, we can also solve the primal and dual problems associated with the pair $(d\mu_h, h)$, where $h$ is any left-invariant metric on $\mathrm{SU}(2)$ and $d\mu_h$ is its volume element.
It turns out that if one searches for solutions of the first-eigenvalue maximization problem within left-invariant metrics on $\mathrm{SU}(2)$, the cases in which the solutions are given by Riemannian metrics are limited to the Berger metrics $h_t$ with $t\leq 1$, up to isometry and scaling, and are therefore rather exceptional.

Establishing the existence of solutions of the primal and dual problems associated with a general pair $(d\mu, h)$ is an important problem.
This problem will be addressed in a future work.
On the other hand, it can be proved that if the first-eigenvalue maximization problem is solved by a {\em Riemannian} metric, then an inflated isometric immersion is obtained as a map consisting of first eigenfunctions of the corresponding Bakry-\'Emery Laplacian (Theorem \ref{nadirashvili-type-thm}).

G\"oring et al.\ \cite{GoeringHelmbergWappler2} introduced the notion of {\em rotational dimension} for a finite graph in terms of solutions to their graph-embedding problem. We define an analogue of it, called the {\em inflation dimension}, for a manifold by employing inflated maps.

\medskip
This paper is organized as follows.
In Section 1, we prepare basic materials on the Bakry-\'Emery Laplacian and symmetric $2$-tensor fields.
In Section 2, for a compact manifold $M$ equipped with a pair $(d\mu, h)$, we formulate an optimization problem concerning maps into $l^2$.
As its Lagrangian dual, we derive the problem of maximizing the first eigenvalue of the Bakry-\'Emery Laplacian.
We observe that the weak duality inequality follows as a consequence of the duality.
We also derive a universal inequality between the two functionals of the respective problems, from which the weak duality inequality can be deduced again.
We write down the equality conditions for this inequality.
We then observe that an isometric minimal immersion into a sphere by first eigenfunctions of the Laplacian is an inflated map.
In Section 3, we present examples of manifolds for which the primal and dual problems introduced in Section 2 can be solved explicitly.
We solve the problems for all flat metrics on the two-dimensional torus and all left-invariant metrics on $\mathrm{SU}(2)$, and in particular, give inflated maps for them.
In Section 4, we prove a Nadirashvili-type theorem when the first-eigenvalue maximization problem is solved by a Riemannian metric.
In Section 5, we define the inflation dimension and give some examples.

\medskip\noindent
{\bf Convention.}\quad A Riemannian metric is primarily defined on the tangent bundle.
In this paper, however, we often regard a Riemannian metric $g$ as being defined on the cotangent bundle.
When doing so, we denote it by $g^*$ in order to avoid confusion.
We also allow $g^*$ to stand for a positive semidefinite metric on the cotangent bundle, in which case $g$ may be defined only partially on the tangent bundle.

\section{Preliminaries}

\subsection{Bakry-\'Emery Laplacian}

Let $M$ be a compact manifold of dimension $n$ equipped with a smooth
volume element $d\mu$ and a Riemannian metric $g$.
Let $\delta_{(d\mu,g)}$ denote the codifferential with respect to $(d\mu,g)$,
defined on one-forms by
$$
\int_M (\delta_{(d\mu,g)} \alpha) u\, d\mu = - \int_M g^*(\alpha, du)\, d\mu,\quad
\alpha\in \Omega^1(M),\,\, u\in C^\infty(M).
$$
The {\em Bakry-\'Emery Laplacian} $\Delta_{(d\mu,g)}\colon C^\infty(M)\to C^\infty(M)$
is defined by
$$
\Delta_{(d\mu,g)} u = \delta_{(d\mu,g)} du,\quad u\in C^\infty(M).
$$
If we write $d\mu_g = e^f d\mu$, where $d\mu_g$ is the Riemannian volume element of $g$,
then
$$
\delta_{(d\mu,g)} \alpha = \delta_g \alpha - g^*(df, \alpha),
$$
where $\delta_g$ is the codifferential with respect to $g$.
It follows that
$$
\Delta_{(d\mu,g)} u = \Delta_g u - g^*(df,du),\quad u\in C^\infty(M),   
$$
where $\Delta_g\colon u\mapsto \delta_g du = g^{ij} (\nabla_{\partial_i}
du)(\partial_j)$ is the Riemannian Laplacian with respect to $g$.
In particular, if $d\mu = c\, d\mu_g$, where $c$ is a positive constant, then
$\Delta_{(d\mu,g)} = \Delta_g$.

Let $\lambda_1(d\mu,g)$ denote the first nonzero eigenvalue of $-\Delta_{(d\mu,g)}$.
It is characterized variationally as
\begin{equation}\label{variationalcharacterization}
\lambda_1(d\mu,g) = \inf \frac{\int_M |du|^2_g\, d\mu}{\int_M u^2\, d\mu},
\end{equation}
where $|du|^2_g = g^*(du, du)$ and $\inf$ is taken over all nonzero $C^1$-functions
$u$ satisfying $\int_M u\, d\mu=0$.
The expression $\mathrm{RQ}(u):= \int_M |du|^2_g\, d\mu / \int_M u^2\, d\mu$
is referred to as the {\em Raileigh quotient}.

Note that the operators $\delta_{(d\mu,g)}$ and $\Delta_{(d\mu,g)}$ are defined,
and the quantity $\lambda_1(d\mu,g)$ is defined and nonnegative, if $g^*$ is only positive
semidefinite on $T^*M$.
Let $\mathcal{M}^*$ denote the space of all positive semidefinite metrics on $T^*M$.
Then the correspondence $g^*\in \mathcal{M}^*\mapsto \delta_{(d\mu,g)}$, and hence
$g^*\in \mathcal{M}^*\mapsto \Delta_{(d\mu,g)}$, is linear.
It follows from \eqref{variationalcharacterization} that the function $g^*\in \mathcal{M}^*
\mapsto \lambda_1(d\mu,g)\in \R$ is concave.
Therefore, the set $\{ g^*\in \mathcal{M}^* \mid \lambda_1(d\mu,g)\geq c \}$ is convex
for any positive number $c$.

Henceforth, we exclusively use $g^*$ instead of $g$ to indicate the dependence on the metric.
For example, we denote $\Delta_{(d\mu,g)}$ and $\lambda_1(d\mu,g)$ by
$\Delta_{(d\mu,g^*)}$ and $\lambda_1(d\mu,g^*)$, respectively.

\subsection{Symmetric $2$-tensors}
Set $S := S^2 T^*M$ (resp.\! $S^*:=S^2 TM$), the bundle of symmetric covariant
(resp.\! contravariant) $2$-tensors.
Let $\mathcal{S}$ (resp.\! $\mathcal{S}^*$) denote the space of $C^\infty$-sections
of $S$ (resp.\! $S^*$).
Note that $\mathcal{M}^*\subset \mathcal{S}^*$.
We denote the (pointwise) pairing between $g^*\in \mathcal{S^*}$ and $h\in \mathcal{S}$
by $( g^*, h )$.
In local coordinates, $( g^*, h ) = g^{ij} h_{ij}$.
We will use the following fact.

\begin{Lemma}\label{positive-semidefinite}
Suppose that $M$ is equipped with a smooth volume element $d\mu$.
Then $h\in \mathcal{S}$ is positive semidefinite at every point of $M$ if and only if
$$
\int_M ( g^*, h )\, d\mu\geq 0
$$
holds for all $g^*\in \mathcal{M}^*$.
\end{Lemma}

\section{Problems}

\subsection{Inflated map}
Let $M$ be a compact manifold of dimension $n$ equipped with a smooth volume
element $d\mu$ and a Riemannian metric $h$. E.g., $d\mu = d\mu_h$,
the volume element of $h$.
Write $\vol(d\mu) = \int_M d\mu$ and $d\mu_1=d\mu/\vol(d\mu)$.
Let $l^2$ denote the Hilbert space of real sequences $\{ a_k \}_{k=1}^\infty$
such that $\sum_{k=1}^\infty {a_k}^2<\infty$.
Let $h_{l^2}$ denote the standard Riemannian metric of $l^2$.

\begin{Problem}\label{problem-emb}
Over all $C^\infty$-maps $\varphi\colon M \to l^2$ satisfying the constraints
\begin{equation}\label{short}
\varphi^* h_{l^2}\leq h
\end{equation}
and
\begin{equation}\label{center}
\int_M \varphi\, d\mu_1 = 0,
\end{equation}
maximize the variance of $\varphi$,
\begin{equation*}\label{variance}
\var(\varphi) := \int_M \|\varphi\|^2\, d\mu_1,
\end{equation*}
where $\|\cdot\|$ is the norm of $l^2$.
\end{Problem}

Note that the zero-average condition \eqref{center} is not restrictive
since it can be achieved by composing a translation in $l^2$ to any given map.
Set
$$
\Var(d\mu,h) := \sup \var(\varphi),
$$
where $\sup$ is taken over all $C^\infty$-maps $\varphi$ satisfying \eqref{short} 
and \eqref{center}.

\begin{Definition}
A solution to Problem \ref{problem-emb} is called an {\em inflated map}.
\end{Definition}

A $C^\infty$-map $\varphi\colon M\to l^2$ satisfying \eqref{short} is called {\em short}.
A short map is $1$-Lipschitz, and vice versa for a smooth map.
An inflated map, if exists, is a globally most expanding map among all locally shrinking
(non-expanding) maps.
It is, therefore, reasonable to expect that an inflated map would be an isometric immersion.
However, as we will observe in Section \ref{examples}, this is not the case in general.

\begin{Remark}\label{problem-emb-RN}
If we restrict maps $\varphi$ to those into $\R^N$, where $N$ is a fixed positive integer,
we denote the corresponding supremum of the functional $\var$ by $\Var^{(N)}(d\mu,h)$.
If $N=1$, $\Var^{(1)}(d\mu,h)$ coincides with (one half of) the {\em observable variance}
$\mathrm{ObsVar}(X)$ of the metric measure space $X=(M, d\mu, d_h)$ introduced by
Nakajima-Shioya \cite{NakajimaShioya}.
They showed, using the Ascoli-Arzel\`a theorem, that there exists a $1$-Lipschitz function
which attains $\mathrm{ObsVar}(X)$.
For general $N$, the same argument shows that there exists a $1$-Lipschitz map
$\varphi\colon M\to \R^N$ which attains $\Var^{(N)}(d\mu,h)$.
Note that such a map $\varphi$ may not be smooth.
\end{Remark}

\subsection{Dual problem}\label{ss-dual-problem}
Let $\mathcal{F}$ denote the set of all $C^\infty$-maps from $M$ into $l^2$
with average equal to the zero vector.
Define a function $L\colon \mathcal{F}\times \mathcal{M}^*\to \R$ by
\begin{equation*}\label{lagrange}
L(\varphi, g^*) := \int_M \| \varphi \|^2\, d\mu_1
+ \int_M ( g^*, h - \varphi^*h_{l^2} )\, d\mu_1.
\end{equation*}
Note the obvious inequality:
\begin{equation}\label{inf-sup}
\sup_{\varphi\in \mathcal{F}} \inf_{g^*\in \mathcal{M}^*} L(\varphi, g^*)\leq \inf_{g^*\in
\mathcal{M}^*}
\sup_{\varphi\in \mathcal{F}} L(\varphi, g^*).
\end{equation}

Fix $\varphi\in \mathcal{F}$.
By Lemma \ref{positive-semidefinite}, $\varphi$ satisfies the constraint \eqref{short}
if and only if
$$
\int_M ( g^*, h - \varphi^*h_{l^2} )\, d\mu_1\geq 0
$$
holds for all $g^*\in \mathcal{M}^*$.
Hence, one has
$$
\inf_{g^*\in \mathcal{M}^*} L(\varphi, g^*) =
\left\{
\begin{array}{cl}
\int_M \| \varphi  \|^2\, d\mu_1 & \mbox{if $\varphi$ satisfies
\eqref{short}},\\
- \infty & \mbox{otherwise}.
\end{array}
\right.
$$
Therefore, the problem on the left-hand side of \eqref{inf-sup} is identical to Problem
\ref{problem-emb}.
In other words, the function $L$ is the Lagrange function associated to Problem \ref{problem-emb}
with $g$ playing the role of Lagrange multiplier.

We shall identify the problem on the right-hand side of \eqref{inf-sup}; It is the dual
problem of Problem \ref{problem-emb} with respect to the Lagrange function $L$.
By noting the equality $( g^*, \varphi^*h_{l^2} ) = \|d\varphi\|^2_{g^*}$,
the function $L$ is rearranged as
$$
L(\varphi, g^*) = \int_M ( g^*, h )\, d\mu_1
+ \left( - \int_M \|d\varphi\|^2_{g^*}\, d\mu_1 + \int_M \| \varphi \|^2\,
d\mu_1 \right).
$$
Fix $g^*\in \mathcal{M}^*$.
Since $\lambda_1(d\mu,g^*)\geq 1$ if and only if
$$
- \int_M \|d\varphi\|^2_{g^*}\, d\mu_1 + \int_M \| \varphi \|^2\,
d\mu_1\leq 0
$$
holds for all $\varphi\in \mathcal{F}$, we obtain
$$
\sup_{\varphi\in \mathcal{F}} L(\varphi,g^*) =
\left\{
\begin{array}{cl}
\int_M ( g^*, h )\, d\mu_1 & \mbox{if $\lambda_1(d\mu,g^*)\geq 1$},\\
\infty & \mbox{otherwise}.
\end{array}
\right.
$$
Therefore, the problem on the right-hand side of \eqref{inf-sup} coincides with the following

\begin{Problem}\label{problem-dual}
Over all $g^*\in \mathcal{M}^*$ satisfying the constraint
\begin{equation}\label{bottom}
\lambda_1(d\mu,g^*)\geq 1,
\end{equation}
minimize
$\int_M ( g^*, h )\, d\mu_1$.
\end{Problem}

As easily seen, the problem remains the same if one replaces the constraint \eqref{bottom}
by the equality $\lambda_1(d\mu,g^*)= 1$.
Thus the problem is equivalent to that of minimizing
\begin{equation}\label{scale-invariant-functional}
\int_M ( g^*, h )\,d\mu_1/\lambda_1(d\mu,g^*)
\end{equation}
(which is invariant under a scaling of $g^*$), or maximizing $\lambda_1(d\mu,g^*)/\int_M ( g^*, h )\, d\mu_1$,
over all $g^*\in \mathcal{M}^*$.

In conclusion, Problem \ref{problem-dual} is equivalent to the following {\em first-eigenvalue
maximization problem}:

\begin{Problem}\label{problem-spec}
Over all $g^*\in \mathcal{M}^*$, maximize  
$$
\lambda_1(d\mu,g^*) \Bigm/ \int_M ( g^*, h )\, d\mu_1.
$$
\end{Problem}

Set
$$
\Lambda_1(d\mu,h) := \sup_{g^*\in \mathcal{M}^*} \lambda_1(d\mu,g^*) \Bigm/
\int_M ( g^*, h )\, d\mu_1.
$$
Clearly,  
$$
\Lambda_1(d\mu,h) = \sup \lambda_1(d\mu,g^*),
$$
where $\sup$ is taken over all $g^*\in \mathcal{M}^*$ satisfying the constraint
\begin{equation*}\label{metric-normalization1}
\int_M ( g^*, h )\, d\mu_1 = 1,
\end{equation*}
or
\begin{equation*}\label{metric-normalization2}
\int_M ( g^*, h )\, d\mu_1\leq 1.
\end{equation*}

As an immediate consequence of the inequality \eqref{inf-sup}, we obtain the {\em weak-duality
inequality}:
\begin{equation}\label{weak-duality}
\Var(d\mu,h)\leq \frac{1}{\Lambda_1(d\mu,h)}.
\end{equation}

\begin{Remark}\label{remark-scaling}
It is manifest that the quantities $\Var(d\mu,h)$ and $\Lambda_1(d\mu,h)$ are invariant under
a scaling of $d\mu$.
On the other hand, for a scaling of $h$, they behave as
$$
\Var(d\mu,ch) = c\, \Var(d\mu,h),\quad \Lambda_1(d\mu,ch) = \frac{1}{c}\, \Lambda_1(d\mu,h),
\quad c>0.
$$
Therefore, the product $\Var(d\mu,h)\cdot \Lambda_1(d\mu,h)$ is invariant under a scaling of $h$.
The inequality \eqref{weak-duality} states that this scale-invariant quantity is less that or equal to $1$.
\end{Remark}

Problem \ref{problem-dual}, which is equivalent to Problem \ref{problem-spec},
is to minimize the linear functional $g^*\mapsto \int_M ( g^*, h )\, d\mu_1$
over the convex set consisting of all $g^*\in \mathcal{M}^*$ such that
$\lambda_1(d\mu,g^*)\geq 1$.
This convexity, combined with an averaging argument, implies the following

\begin{Proposition}\label{invariant-solution}
Let $M$ be a compact manifold equipped with a smooth volume element $d\mu$
and a Riemannian metric $h$.
Let $G$ be the automorphism group of $(M,d\mu,h)$, that is, the set of all diffeomorphisms
of $M$ which preserve $d\mu$ and $h$.
If there exists a solution to Problem \ref{problem-spec}, then there also exists a $G$-invariant
solution to Problem \ref{problem-spec}.
\end{Proposition}

\subsection{Direct route to the weak duality}\label{direct}
While the idea of Lagrangian duality was useful to identify the dual problem of Problem \ref{problem-emb},
the ineqaulity \eqref{weak-duality} can be proved directly without referring to the duality.
In fact, if $\varphi\in \mathcal{F}$ satisfies \eqref{short} and $g^*\in \mathcal{M}^*$,
then by taking the pairing of $g^*$ and the both sides of \eqref{short}, we obtain
$
\| d\varphi \|_{g^*}^2 \leq ( g^*, h ).
$
This and \eqref{variationalcharacterization} imply
$$
\int_M \| \varphi \|^2\, d\mu
\leq \frac{1}{\lambda_1(d\mu,g^*)} \int_M \| d\varphi \|_{g^*}^2\, d\mu
\leq \frac{1}{\lambda_1(d\mu,g^*)} \int_M ( g^*, h )\, d\mu,
$$
from which \eqref{weak-duality} follows.
Both the inequalities become equalities if and only if $\varphi$ consists
of functions minimizing the Rayleigh quotient $\mathrm{RQ}$ over all functions
with mean value zero
and $( g^*, h - \varphi^* h_{l^2} ) \equiv 0$.
Note that the latter condition implies $\varphi^* h_{l^2} \equiv h$, that is,
$\varphi$ is an isometric immersion with respect to $h$, if $g^*$ is positive
definite at every point of $M$.
However, this is not the case in general as mentioned in Introduction.

In summary, we have shown

\begin{Proposition}\label{variance-inequality}
Let $M$ be a compact manifold equipped with a smooth volume element $d\mu$ and a Riemannian metric $h$.
Then, for any $\varphi\in \mathcal{F}$ satisfying \eqref{short} and $g^*\in \mathcal{M}^*$,
we have
\begin{equation}\label{variance-lambda-inequality}
\int_M \| \varphi \|^2\, d\mu_1\leq
\int_M ( g^*, h )\, d\mu_1 \Bigm/ \lambda_1(d\mu,g^*),
\end{equation}
which implies the weak-duality inequality \eqref{weak-duality}.
Equality holds in \eqref{variance-lambda-inequality} if and only if the map
$\varphi$ consists of functions minimizing the Rayleigh quotient $\mathrm{RQ}$
over all functions with mean value zero and
\begin{equation}\label{condition-equality}
( g^*, h - \varphi^* h_{l^2} ) \equiv 0
\end{equation}
holds.
If $g^*$ is positive definite at every point of $M$, then these equality conditions
imply that the map $\varphi$ consists of first eigenfunctions of $-\Delta_{(d\mu,g^*)}$
and $\varphi$ is an isometric immersion with respect to $h$.
\end{Proposition}

Note that in the last statement of the proposition, the image of $\varphi$ lies in a finite
dimensional subspace of $l^2$ since the multiplicity of the first eigenvalue is finite.

\medskip
Suppose $d\mu=d\mu_h$ is chosen initially.
Then the equality holds in \eqref{variance-lambda-inequality} with $g^*=h^*$
if and only if the map $\varphi$ is an $h$-isometric minimal immersion into a
finite-dimensional sphere of radius $\sqrt{n/\lambda_1(h)}$, where $n = \dim M$,
by first eigenfunctions of $-\Delta_h$.
In fact, under $g^*=h^*$, the equality condition for \eqref{variance-lambda-inequality}
reduces to the {\em Takahashi condition}:\,\, the map $\varphi$ is an $h$-isometric immersion
into a Euclidean space by first eigenfunctions of $-\Delta_h$.
Then by the Takahashi theorem \cite{Takahashi}, this condition implies that
the image of $\varphi$ lies in a sphere of radius $\sqrt{n/\lambda_1(h)}$ and $\varphi$
is minimal as an immersion into the sphere.
In particular, we have

\begin{Proposition}\label{minimal-to-sphere}
Let $(M,h)$ be a compact Riemannian manifold and let $\varphi$ be an isometric minimal
immersion of $(M, h)$ into a sphere (of some radius) by first eigenfunctions of $-\Delta_h$.
Then $\varphi$ is an inflated map with respect to $(d\mu_h, h)$, and $g^*=h^*$ provides a solution
to Problem \ref{problem-spec} associated with $(d\mu_h, h)$.
\end{Proposition}

In fact, we can characterize the metrics $h$ for which Problem \ref{problem-spec} associated
with $(d\mu_h, h)$ is {\em self-solvable}, that is, $g^*=h^*$ gives a solution, as those admitting
a minimal immersion of the above kind into a sphere (Corollary \ref{self-solvable}).   

\begin{Remark}\label{example-minimal-in-sphere}
Many examples of compact Riemannian manifolds $(M,h)$ that admit an isometric minimal
immersion into a sphere by first eigenfunctions are known.
These include all compact isotropy irreducible Riemannian homogeneous spaces \cite{Takahashi}.
Yau \cite{Yau} conjectured that the inclusion map of any compact embedded minimal
hypersurface in a sphere was such an immersion.
Choe and Soret \cite{ChoeSoret} verified Yau's conjecture for the Lawson surfaces and
the surfaces constructed by Karcher-Pinkall-Sterling in $S^3$.
Muto \cite{Muto2}, Tang and Yan \cite{TangYan, TangXieYan} verified it for compact minimal
isoparametric hypersurfaces in spheres.

On the other hand, there is a result indicating that such metrics should be
sufficiently rare.
Indeed, on the torus $T^2$, there are only two metrics, up to scaling, that admit
such an immersion \cite{ElSoufiIlias}.
\end{Remark}

\subsection{Linearization}\label{linearization}
As mentioned in Subsection \ref{ss-dual-problem}, Problem \ref{problem-dual} is
to minimize the linear functional $g^*\mapsto \int_M ( g^*, h )\, d\mu_1$
over the convex set $\{ g^*\in \mathcal{M}^* \mid \lambda_1(d\mu,g^*)\geq 1 \}$.
On the other hand, the functional $\varphi\mapsto \var(\varphi)$ of
Problem \ref{problem-emb} is quadratic.
We can, however, linearize it by changing variables.
We begin by observing that one can associate a positive symmetric smoothing operator
to a $C^\infty$-map into $l^2$.
For $\varphi=(u_k)_{k=1}^\infty\in \mathcal{F}$, define a $C^\infty$-function
$K_\varphi$ on $M\times M$ by
\begin{equation}\label{kernel}
K_\varphi(p,q)= \sum_{k=1}^\infty u_k(p) u_k(q),\quad (p,q)\in M\times M.
\end{equation}
Note that $K_\varphi$ has zero $d\mu_1$-average with respect to both the variables.
Let $L^2_0(d\mu_1)$ denote the space of functions in $L^2(d\mu_1)$ with zero $d\mu_1$-average.
Then the operator $X_\varphi\colon L^2_0(d\mu_1)\to L^2_0(d\mu_1)$ defined by
\begin{equation}\label{smoothing}
X_\varphi(f) = \int_M K_\varphi(\cdot,q) f(q)\, d\mu_1(q),\quad f\in
L^2_0(d\mu_1)
\end{equation}
is a positive symmetric $C^\infty$-smoothing operator with kernel function $K_\varphi$.

We introduce the following terminologies.

\begin{Definition}
Let $\varphi=(u_k)_{k=1}^\infty\in \mathcal{F}$, a $C^\infty$-map into $l^2$ with average
equal to the zero vector.
We call the kernel function $K_\varphi$ defined by \eqref{kernel} and the smoothing
operator $X_\varphi\colon L^2_0(d\mu_1)\to L^2_0(d\mu_1)$ defined by \eqref{smoothing} the
{\em Gram kernel} and the {\em Gram opoerator} associated with $\varphi$, respectively.
Let $\mathcal{X} := \mathcal{X}(d\mu_1)$ denote the set of all Gram operators.
\end{Definition}

Let $\mathcal{T}_\mathcal{X}$ denote the class of symmetric, densely defined operators
$T\colon L^2_0(d\mu_1)\to L^2_0(d\mu_1)$ that satisfy the following condition:\,\,
For any $X\in \mathcal{X}$,
$$
\langle T, X \rangle := \sum_{i=1}^\infty \langle T(e_i), X(e_i) \rangle
$$
converges for every complete orthonormal system $\{ e_i \}_{i=1}^\infty$ of
$L^2_0(d\mu_1)$ consisting of functions in the domain of $T$.
The identity operator $I$ and the positive Bakry-\'Emery Laplacian $-\Delta_{(d\mu, g^*)}$
lies in $\mathcal{T}_\mathcal{X}$.
In fact, for $\varphi\in \mathcal{F}$, it is easy to verify
\begin{equation*}
\langle I, X_\varphi \rangle = \int_M \| \varphi \|^2\, d\mu_1
\end{equation*}
and
\begin{equation}\label{pairing}
\langle -\Delta_{(d\mu,g^*)}, X_\varphi \rangle  = \int_M ( g^*, \varphi^* h_{l^2} )\, d\mu_1.
\end{equation}
Since the right-hand side of \eqref{pairing} is the pairing of
$g^*\in \mathcal{S}^*$ and $\varphi^* h_{l^2}\in \mathcal{S}$, we may write
$$
\langle -\Delta_{(d\mu,\bullet)}, X_\varphi \rangle
= \varphi^* h_{l^2}
$$
formally and regard $\langle -\Delta_{(d\mu,\bullet)}, X \rangle\in \mathcal{S}$
for $X\in \mathcal{X}$.
With this understood, Problem \ref{problem-emb} is reformulated in terms of the Gram operator
as follows.

\begin{Problem}\label{problem-op}
Over all $X\in \mathcal{X}$ satisfying $\langle -\Delta_{(d\mu,\bullet)}, X \rangle\leq h$,
maximize $\langle I, X \rangle$.
\end{Problem}

Observe that one maximizes a linear functional over a convex set of Gram operators.

\begin{Remark}\label{goering}
Let $G=(V,E)$ be a finite connected graph with vertex set $V$ and edge set $E$, equipped with
a vertex weight $s\colon V\to \R_{>0}$ and an edge length $l\colon E\to \R_{>0}$.
G\"oring-Helmberg-Wappler \cite{GoeringHelmbergWappler2} (see also \cite{GoeringHelmbergWappler1}) introduced a dual pair of optimization problems concerning graph embeddings and first eigenvalue of a certain graph Laplacian.  
The variable of their first-eigenvalue maximization problem is the edge weight $w\colon E\to \R_{\geq 0}$.
We formulate Problems \ref{problem-emb} and \ref{problem-spec} by interpreting the volume element $d\mu$, the Riemannian metric $h$, and the metric $g^*$ on the cotangent bundle as analogues of $s$, $l^2$ (squared length), and $w$, respectively.
As the resulting first-eigenvalue maximization problem for a manifold is rather different from the classical one involving the Riemannian Laplacian, a natural question arises:\,\, What is an analogue of the maximization problem for the first eigenvalue of the Riemannian Laplacian in the graph theory?
Gomyou and the present author \cite{GomyouNayatani1, GomyouNayatani2} formulated and studied a new first-eigenvalue maximization problem for a graph $G$ equipped with an edge length $l$.
The relevant Laplacian is the {\em Fujiwara Laplacian} \cite{Fujiwara}, which is a graph Laplacian defined only in terms of $l$.

It is worth noting that G\"oring et al.\ had also employed linearization in order to treat their dual pair of problems in the framework of semidefinite programming.
\end{Remark}

\section{Examples}\label{examples}
In Subsection \ref{direct}, we discussed examples for which Problems \ref{problem-emb} and
\ref{problem-spec} are solved in a trivial manner.
In this section, we solve these problems explicitly for some nontrivial examples.

\subsection{Flat tori}\label{ss-flat-tori}

The simplest nontrivial examples are provided by flat tori.
As mentioned in Remark \ref{example-minimal-in-sphere}, exactly two metrics on $T^2$,
up to scaling, admit an isometric minimal immersion into a sphere by first eigenfunctions.
These metrics are flat and it turns out that they can be used to solve Problems
\ref{problem-emb} and \ref{problem-spec} associated with any flat metric on $T^2$.
So we shall examine these metrics.
We begin by reviewing basic facts about flat metrics on $T^2$ and their spectral calculus
(see \cite{BergerGauduchonMazet}).

If $\Gamma$ is a lattice in $\R^2$, acting on $\R^2$ by translations, then the quotient space
$\R^2/\Gamma$ is diffeomorphic to $T^2$ and the Euclidean metric of $\R^2$ descends to a flat
metric on $\R^2/\Gamma$.
Any flat metric on $T^2$ is obtained in this way.   
The eigenvalues and associated
eigenfunctions of the corresponding Laplacian are given by $4\pi^2 \|\gamma^*\|^2$ and
$\exp (2\pi i \langle \gamma^*, x \rangle)$, respectively,
where $\gamma^*$ are elements of the dual lattice
$$
\Gamma^* = \{ \gamma^* \mid \mbox{$\langle \gamma^*, x \rangle \in \Z$ for all $x\in \Gamma$} \}
$$
of $\Gamma$.
By identifying lattices which induce isometric metrics up to scaling, we may assume that
$\Gamma = \Z (1,0)\oplus \Z (a,b)$, where $0\leq a\leq 1/2$, $b\geq \sqrt{1-a^2}$.
Let $h$ denote the corresponding flat metric.
Then $\Gamma^* = \Z (1,-a/b)\oplus \Z (0,1/b)$, and the first nonzero eigenvalue $\lambda_1(h)$
of $-\Delta_h$ is given by
$$
\lambda_1(h) = 4\pi^2\|(0,1/b)\|^2 = \frac{4\pi^2}{b^2}
$$
with multiplicity
$$
\left\{\begin{array}{cl} 6, & \mbox{if $(a,b)=(1/2, \sqrt{3}/2)$,}\\
4, & \mbox{if $b=\sqrt{1-a^2}$ and $a\neq 1/2$,}\\
2, & \mbox{if $b>\sqrt{1-a^2}$.}\end{array}\right.
$$

The two distinguished metrics mentioned above are the flat metrics corresponding to the two
highly-symmetric lattices: the {\em square lattice} $\Gamma_{\mathrm{SQ}}$ and the
{\em equilateral lattice} $\Gamma_{\mathrm{EL}}$ for the choices $(a,b)=(0,1)$ and
$(a,b)=(1/2,\sqrt{3}/2)$, respectively.
We examine closely the flat metric, denoted by $h_{\mathrm{EL}}$, corresponding to the lattice
$\Gamma_{\mathrm{EL}}$.
The first nonzero eigenvalue is $\lambda_1(h_{\mathrm{EL}}) = 16\pi^2/3$ with multiplicity $6$,
and a basis for the first eigenspace is given by the real and imaginary parts of the functions
\begin{equation}\label{EL-eigenfunction}
\exp (2\pi i (x-y/\sqrt{3})),\quad \exp (4\pi i y/\sqrt{3}),\quad \exp (2\pi i (x+y/\sqrt{3})).
\end{equation}
One can verify that the map
\begin{align*}
& \varphi\colon \mbox{$(x,y)$ $\mathrm{mod}$ $\Gamma_{\mathrm{EL}}$}
\in \R^2/\Gamma_{\mathrm{EL}}\mapsto\\
& \phantom{\varphi\colon \mbox{$(x,y)$}}
\frac{1}{\sqrt{8\pi^2}}
\big( \exp (2\pi i (x-y/\sqrt{3})), \exp (4\pi i y/\sqrt{3}), \exp (2\pi i (x+y/\sqrt{3}))
\big)\\
& \phantom{\varphi\colon \mbox{$(x,y)$ $\mathrm{mod}$ $\Gamma_{\mathrm{EL}}$}}
\in S^5\big(\sqrt{3/8\pi^2}\big)\subset \C^3
\end{align*}
is isometric with respect to $h_{\mathrm{EL}}$ and is one-to-one.
Thus, $\varphi$ is an isometric minimal embedding into a sphere whose coordinates are
first eigenfunctions.
By Proposition \ref{minimal-to-sphere}, the map $\varphi$ is an inflated map with respect
to $(d\mu_{h_{\mathrm{EL}}},h_{\mathrm{EL}})$, and the metric $g^*=h_{\mathrm{EL}}^*$
is a solution to Problems \ref{problem-spec} associated with
$(d\mu_{h_{\mathrm{EL}}},h_{\mathrm{EL}})$.
The flat metric $h_{\mathrm{SQ}}$ corresponding to the square lattice $\Gamma_{\mathrm{SQ}}$
is simpler. So we state conclusions only.
The first nonzero eigenvalue is $\lambda_1(h_{\mathrm{SQ}}) = 4\pi^2$ with multiplicity $4$,
and the map
$$
\varphi\colon \mbox{$(x,y)$ $\mathrm{mod}$ $\Gamma_{\mathrm{SQ}}$}\in \R^2/\Gamma_{\mathrm{SQ}}
\mapsto \frac{1}{2\pi}
\big(\exp (2\pi i x), \exp (2\pi i y)\big)\in S^3\big(1/\sqrt{2}\pi\big)\subset \C^2
$$
is an isometric minimal embedding into a sphere whose coordinates are first eigenfunctions.
The image of this embedding is called the {\em Clifford torus}.
Again, $\varphi$ is an inflated map with respect to $(d\mu_{h_{\mathrm{SQ}}},h_{\mathrm{SQ}})$,
and $g^*=h_{\mathrm{SQ}}^*$ is a solution to Problems \ref{problem-spec} associated with
$(d\mu_{h_{\mathrm{SQ}}},h_{\mathrm{SQ}})$.

We now use these two metrics to solve Problems \ref{problem-emb}
and \ref{problem-spec} for {\em every}\, flat metric on $T^2$.
For $0\leq a\leq 1/2$ and $b\geq \sqrt{1-a^2}$,
let $h_{a,b}$ denote the flat metric of $\R^2/\Z(1,0)\oplus \Z(a,b)$, pulled back to
$\R^2/\Z^2$
via the linear isomorphism fixing $(1,0)$ and sending $(0,1)$ to $(a,b)$.
Thus,
$$
h_{a,b} = dx^2 + 2adxdy + (a^2+b^2) dy^2.
$$
We will solve Problems \ref{problem-emb} and \ref{problem-spec} associated with
$(d\mu_{h_{a,b}}, h_{a,b})$.

First, we use the metric $h_{\mathrm{EL}} = dx^2 + dxdy + dy^2$.
A basis for the first eigenspace of $-\Delta_{h_{\mathrm{EL}}}$ is given by
the functions \eqref{EL-eigenfunction} pulled back to $\R^2/\Z^2$, that is,
$$
\exp (2\pi i x),\quad \exp (2\pi i y),\quad \exp (2\pi i (x+y)).
$$
Consider the map
\begin{align*}
& \psi_{p,q,r}\colon \mbox{$(x,y)$ $\mathrm{mod}$ $\Z^2$}\in \R^2/\Z^2\mapsto\\
&\phantom{\psi_{p,q,r}\colon } \frac{1}{\sqrt{8\pi^2}}
\big( p\, \exp (2\pi i x), q\, \exp (2\pi i y), r\, \exp (2\pi i (x+y)) \big) \in \C^3
\end{align*}
consisting of first eigenfunctions of $-\Delta_{h_{\mathrm{EL}}}$, where $p,q,r\geq 0$.
It satisfies ${\psi_{p,q,r}}^* g_{\C^3} = h_{a,b}$
if and only if  
$$
p = \sqrt{2(1-a)},\quad q = \sqrt{2(a^2+b^2-a)},\quad r = \sqrt{2a}.
$$
Denote the map $\psi_{p,q,r}$ with these choices of $p,q,r$ by $\varphi_{a,b}$.
Then the map $\varphi=\varphi_{a,b}$ and the metric $g^*=h_{\mathrm{EL}}^*$ satisfy
the equality conditions for \eqref{variance-lambda-inequality} as in Proposition
\ref{variance-inequality}.
Therefore, $\varphi=\varphi_{a,b}$ and $g^*=h_{\mathrm{EL}}^*$ are solutions to Problems
\ref{problem-emb} and \ref{problem-spec} associated with $(d\mu_{h_{a,b}}, h_{a,b})$,
respectively.
Here, it should be noted that the volume elements $d\mu_{h_{a,b}} = b\, dxdy$ differ only
by a multiplicative constant,
and therefore, $\Delta_{h_{\mathrm{EL}}} = \Delta_{(d\mu_{h_{a,b}}, h_{\mathrm{EL}})}$.
Observe that the maps $\varphi_{a,b}$ vary smoothly with the
parameters $a$, $b$ and degenerate to maps into $\R^4$ as $a\to 0$, that is, as the lattice
approaches to a rectangular one.

Secondly, we use the metric $h_{\mathrm{SQ}}=dx^2+dy^2$.
Consider the map
$$
\psi_{p,q}\colon \mbox{$(x,y)$ $\mathrm{mod}$ $\Z^2$}\in
\R^2/\Z^2\mapsto \frac{1}{2\pi} \big(p\, \exp (2\pi i x), q\,
\exp (2\pi i y)\big) \in \C^2
$$
consisting of first eigenfunctions of $-\Delta_{h_{\mathrm{SQ}}}$.
Since we have ${\psi_{p,q}}^* g_{\C^2} = p^2 dx^2 + q^2 dy^2$, these metrics can express
only flat metrics of rectangular tori.
In fact, choosing $p=1$, $q=b$, the map $\varphi_b:=\psi_{1,b}$ satisfies
${\varphi_b}^* g_{\C^2} = h_{0,b}$.
We conclude that the map $\varphi=\varphi_b$ and the metric $g^*=h_{\mathrm{SQ}}^*$ are
solutions to Problems \ref{problem-emb} and \ref{problem-spec} associated with
$(d\mu_{h_{0,b}}, h_{0,b})$, respectively.

While the maps $\varphi_{0,b}$ and $\varphi_b$ are the same, the metrics $h_{\mathrm{EL}}^*$
and $h_{\mathrm{SQ}}^*$ provide two distinct solutions to Problem \ref{problem-spec}.
Then the convexity of Problem \ref{problem-spec} as mentioned in Subsection \ref{ss-dual-problem}
predicts that there should exist a one-parameter family of metrics
joining these two metrics such that all members of the family are solutions to the problem.
In fact, the flat metrics $h_{a,\sqrt{1-a^2}}$ with $0\leq a\leq 1/2$
provide such a family.
To see this, note that the functions $\exp (2\pi i x)$ and $\exp (2\pi i y)$ are common
first eigenfucntions of $-\Delta_{h_{a,\sqrt{1-a^2}}}$, $0\leq a\leq 1/2$, and therefore,
the map $\varphi_b$ consists of first eigenfunctions of $-\Delta_{h_{a,\sqrt{1-a^2}}}$
for all $0\leq a\leq 1/2$.
Thus, the metrics $g^*=h_{a,\sqrt{1-a^2}}^*$, $0\leq a\leq 1/2$, are solutions to
Problems \ref{problem-spec} associated with $(d\mu_{h_{0,b}}, h_{0,b})$.

We record what we have observed as

\begin{Proposition}
For any flat metric $h$ on $T^2$, there exists an isometric embedding $\varphi\colon (T^2, h)\to
\R^6$ by first eigenfunctions of the equilateral flat metric $h_{\mathrm{EL}}$.
In particular, the map $\varphi$ is an inflated map with respect to $(d\mu_h, h)$,
and the metric $g^*=h_{\mathrm{EL}}^*$ is a solution to Problem \ref{problem-spec} associated with
$(d\mu_h, h)$.

If $h$ is a rectangular flat metric, that is, a flat metric corresponding to a rectangular lattice,
then the image of $\varphi$ lies in $\R^4$ and the square flat metric $h_{\mathrm{SQ}}$
provides another solution $h_{\mathrm{SQ}}^*$ to Problem \ref{problem-spec}.  
Moreover, the metrics $h_{a,\sqrt{1-a^2}}^*$, $0\leq a\leq 1/2$, joining
$h_{\mathrm{EL}}^*$ and $h_{\mathrm{SQ}}^*$, all provide solutions to
Problems \ref{problem-spec}.
\end{Proposition}

\subsection{Berger spheres}\label{berger-S3}

The special unitary group of degree $2$,
$$
\mathrm{SU}(2) = \left\{ A(\zeta_1, \zeta_2) := \left(\begin{array}{cr} \zeta_1 & -\overline{\zeta_2}\\
\zeta_2& \overline{\zeta_1} \end{array} \right) \mid \zeta_1\overline{\zeta_1}
+\zeta_2\overline{\zeta_2}=1 \right\},
$$
is a compact Lie group, diffeomorphic to $S^3$.
The real one-forms $\sigma_1$, $\sigma_2$, $\sigma_3$ given by
$$
\sigma_1 + i \sigma_2 = i(\zeta_2 d\zeta_1-\zeta_1 d\zeta_2),\quad
\sigma_3 = -i ( \overline{\zeta_1} d\zeta_1 + \overline{\zeta_2} d\zeta_2)
$$
are left-invariant, and satisfy $d\sigma_a = 2\sigma_b\wedge \sigma_c$, where $(a,b,c)$
is a cyclic permutation of $(1,2,3)$.
Then
\begin{align*}
& E_1 = -i \left( \overline{\zeta_2} \frac{\partial}{\partial \zeta_1}
- \overline{\zeta_1} \frac{\partial}{\partial \zeta_2}
- \zeta_2 \frac{\partial}{\partial \overline{\zeta_1}}
+ \zeta_1 \frac{\partial}{\partial \overline{\zeta_2}} \right),\\
& E_2 = \overline{\zeta_2} \frac{\partial}{\partial \zeta_1}
- \overline{\zeta_1} \frac{\partial}{\partial \zeta_2}
+ \zeta_2 \frac{\partial}{\partial \overline{\zeta_1}}
- \zeta_1 \frac{\partial}{\partial \overline{\zeta_2}},\\
& E_3 = i \left( \zeta_1 \frac{\partial}{\partial \zeta_1}
+ \zeta_2 \frac{\partial}{\partial \zeta_2}
- \overline{\zeta_1} \frac{\partial}{\partial \overline{\zeta_1}}
- \overline{\zeta_2} \frac{\partial}{\partial \overline{\zeta_2}} \right)
\end{align*}
are the left-invariant vector fields such that
$\sigma_a(E_b)=\delta_{ab}$, and satisfy $[E_a,E_b] = -2 E_c$, where $(a,b,c)$
is as above.

For $t>0$, the Berger metric $h_t$ is the left-invariant metric on $\mathrm{SU}(2)$
given by
$$
h_t = {\sigma_1}^2 + {\sigma_2}^2 + t^2 {\sigma_3}^2.
$$
We will call it {\em small} if $t\leq 1$ and {\em large} if $t\geq 1$.
Urakawa \cite{Urakawa} calculated the first eigenvalue $\lambda_1(h_t)$ of the Laplacian
of $h_t$ and its multiplicity $m_1(h_t)$:
$$
\lambda_1(h_t) =
\left\{\begin{array}{cl} 8, & t\leq 1/\sqrt{6},\\ 2+1/t^2, & t\geq 1/\sqrt{6}
\end{array} \right.
$$
and
$$
m_1(h_t) = \left\{\begin{array}{cl} 3, & t<1/\sqrt{6},\\ 7, & t=1/\sqrt{6},\\
4, & t>1/\sqrt{6}. \end{array} \right.
$$
He also identified the corresponding eigenfunctions in terms of representaion theory.
More recently, Lauret \cite{Lauret} extended these results of Urakawa to all
left-invariant metrics on $\mathrm{SU}(2)$.
The following lemma verifies their results in an elementary manner.

\begin{Lemma}[\cite{Lauret}]\label{eigensystem-su2}
The positive Laplacian of the left-invariant metric
\begin{equation}\label{left-invariant}
h_{a,b,c} = \frac{1}{a} {\sigma_1}^2 + \frac{1}{b} {\sigma_2}^2 + \frac{1}{c} {\sigma_3}^2,
\quad a,b,c>0,
\end{equation}
on $\mathrm{SU}(2)$ is given by
\begin{equation}\label{berger-laplacian}
- \Delta_{h_{a,b,c}} = - a (E_1)^2 - b (E_2)^2 - c (E_3)^2.
\end{equation}
It has the following eigenvalues and eigenfunctions:
$$
\begin{array}{ll}
a+b+c, & \zeta_1,\,\, \zeta_2;\\
4(b+c), & \zeta_1^2 - \overline{\zeta_2}^2,\,\, \zeta_1\zeta_2 + \overline{\zeta_1}\overline{\zeta_2};\\
4(c+a), & \zeta_1^2 + \overline{\zeta_2}^2,\,\, \zeta_1\zeta_2 - \overline{\zeta_1}\overline{\zeta_2};\\
4(a+b), & \zeta_1 \overline{\zeta_2},\,\, \zeta_1 \overline{\zeta_1} - \zeta_2 \overline{\zeta_2}.
\end{array}
$$
\end{Lemma}

As remarked in \cite{Lauret}, any left-invariant metric on $\mathrm{SU}(2)$ is isometric to
$h_{a,b,c}$ for some $a$, $b$, $c>0$.

\begin{proof}
Note that $( \sqrt{a} E_1, \sqrt{b} E_2, \sqrt{c} E_3)$ is an orthonormal frame field for $h_{a,b,c}$.
The Koszul formula for the Levi-Civita connection of a left-invariant metric reads
$$
2 \langle \nabla_XY, Z\rangle = \langle [X,Y], Z \rangle - \langle [X,Z], Y \rangle
- \langle [Y,Z], X \rangle,
$$
and we find $\nabla_{E_i} E_i = 0$, $1\leq i\leq 3$.
Therefore, \eqref{berger-laplacian} follows.
One can now verify the second assertion by direct computation.
\end{proof}

Let $\varphi_0\colon \mathrm{SU}(2)\to \R^4$ be the tautological embedding
$A(\zeta_1, \zeta_2)\mapsto (\zeta_1, \zeta_2)$ and
$\varphi_1\colon \mathrm{SU}(2)\to \R^3$ the Hopf map $A(\zeta_1, \zeta_2)\mapsto
( \zeta_1 \overline{\zeta_2},
\frac{1}{2} (\zeta_1 \overline{\zeta_1} - \zeta_2 \overline{\zeta_2}) )$.
For $p,q\geq 0$, consider the map $\varphi_{p,q} = p\,\varphi_0 \oplus q\,\varphi_1
\colon \mathrm{SU}(2)\to \R^7$.
Then since $\varphi_0^* h_{\R^4} = {\sigma_1}^2 + {\sigma_2}^2 + {\sigma_3}^2$
and $\varphi_1^* h_{\R^3} = {\sigma_1}^2 + {\sigma_2}^2$, we obtain
$$
\varphi_{p,q}^* h_{\R^7} = (p^2+q^2) \left( {\sigma_1}^2 + {\sigma_2}^2 \right)
+ p^2 {\sigma_3}^2.
$$
Suppose that $t\leq 1$. Then, choosing $(p, q)=\left(t, \sqrt{1-t^2} \right)$, we obtain
an isometric embedding $\varphi_{t, \sqrt{1-t^2}}\colon (\mathrm{SU}(2), h_t)\to \R^7$
by first eigenfunctions of $h_{1/\sqrt{6}}$.
Note that the unique left-invariant volume element
of unit total volume is the common normalized volume element of the metrics $h_t$,
and therefore, $\Delta_{h_{1/\sqrt{6}}}=\Delta_{(d\mu_{h_t}, h_{1/\sqrt{6}}^*)}$.
Thus, the embedding $\varphi_{t, \sqrt{1-t^2}}$ and the metric $h_{1/\sqrt{6}}^*$
satisfy the equality conditions for \eqref{variance-lambda-inequality} as in
Proposition \ref{variance-inequality}.
Therefore, $\varphi_{t, \sqrt{1-t^2}}$ and $h_{1/\sqrt{6}}^*$ give solutions to Problems
\ref{problem-emb} and \ref{problem-spec} associated with $(d\mu_{h_t}, h_t)$, respectively.

\begin{Remark}
The map $\varphi_{t, \sqrt{1-t^2}}$ is not essentially new.
For $t<1$, the Berger sphere $(\mathrm{SU}(2), h_t)$ is isometric to the geodesic sphere
$$
M_t := \left\{ (z_1:z_2:z_3)\in \C P^2 \mid z_1\overline{z_1} + z_2\overline{z_2}
= (t^{-2}-1) z_0^2 \right\}
$$
in $\C P^2$, up to scaling, where $z_1, z_2, z_3$ are the homogeneous coordinates
of $\C P^2$.
As a special instance of Remark \ref{example-minimal-in-sphere}, $\C P^2$ admits an
isometric embedding into $\R^8$ by its first eigenfunctions.
Restricting this map to $M_t$ and pulling back to $\mathrm{SU}(2)$ give the map
$\varphi_{t, \sqrt{1-t^2}}$, up to a dilation of $\R^7$.
\end{Remark}

In the case when $t\geq 1$, we will verify that the metric
$h_\infty^* = E_1\otimes E_1 + E_2\otimes E_2$,
which is only positive semidefinite on the cotangent bundle, is a
solution to Problem \ref{problem-spec}.
The corresponding Bakry-\'Emery Laplacian $\Delta_{(d\mu_{h_t}, h_\infty^*)}$
coincides with the sub-Laplacian of the unit sphere $S^3\subset \C^2$
endowed with the standard pair of strongly pseudoconvex CR structure and
compatible contact form (see \cite{Greenleaf}).
The eigenvalues and eigenfunctions of the latter operator are computed in
\cite{CowlingKlimaSikora, Stanton}.
It follows in particular that $\lambda_1(d\mu_{h_t}, h_\infty^*)=2$ with eigenfunctions
$\zeta_1$, $\zeta_2$.
Since $t\geq 1$, the tautological embedding $\varphi_0$, which is a map consisting of
the first eigenfunctions $\zeta_1$, $\zeta_2$ of $- \Delta_{(d\mu_{h_t}, h_\infty^*)}$,
satisfies the constraint \eqref{short} of Problem \ref{problem-emb} associated with
$(d\mu_{h_t}, h_t)$.
We also have
$$
( h_\infty^*, h_t - {\varphi_0}^* h_{\R^4} ) = ( E_1\otimes E_1 + E_2\otimes E_2,
(t^2-1) \sigma_3^2 ) \equiv 0.
$$
Thus, the map $\varphi_0$ and the metric $h_\infty^*$ satisfy the equality conditions
for \eqref{variance-lambda-inequality} as in Proposition \ref{variance-inequality}.
Therefore, $\varphi_0$ and $h_\infty^*$  give solutions to Problem \ref{problem-emb}
and Problem \ref{problem-spec} associated with $(d\mu_{h_t}, h_t)$, respectively.

As for Problem \ref{problem-emb} associated with $(d\mu_{h_1}, h_1)$, we have solutions
$h_1^*$, $h_{1/\sqrt{6}}^*$ and $h_\infty^*$ to it.
So, by the same reason as in Subsection \ref{ss-flat-tori}, there should exist a convex
family of solution metrics containing these three ones.
In fact, the metrics $h_t^*$, $1/\sqrt{6}\leq t\leq \infty$, constitute such a family,
since the map $\varphi_0$ is isometric with respect to $h_1$ and consists of
common first eigenfunctions $\zeta_1$, $\zeta_2$ of $-\Delta_{(d\mu_{h_1}, h_t^*)}$
with $t$ in the above range.
Thus, the map $\varphi_0$ and the metrics $h_t^*$, $1/\sqrt{6}\leq t\leq \infty$,
satisfy the equality conditions for \eqref{variance-lambda-inequality}
as in Proposition \ref{variance-inequality}.

Summarizing what we have observed, we obtain the following

\begin{Proposition}\label{Berger-summary}
The small Berger spheres $(\mathrm{SU}(2), h_t)$, $t\leq 1$, admit an isometric embedding
$\varphi_t\colon \mathrm{SU}(2) \to \R^7$ by first eigenfunctions of
$-\Delta_{d\mu_{h_t}, h_{1/\sqrt{6}}^*}$.
Therefore, the map $\varphi_t$ is an inflated map with respect to $(d\mu_{h_t,} h_t)$,
and the metric $h_{1/\sqrt{6}}^*$ is a solution to Problem \ref{problem-spec}
associated with $(d\mu_{h_t,} h_t)$.

For the large Berger spheres $(\mathrm{SU}(2), h_t)$, $t\geq 1$, the tautological embedding
$\varphi_0\colon \mathrm{SU}(2) \to \R^4$ is a common short map
by first eigenfunctions of $-\Delta_{(d\mu_{h_t}, h_\infty^*)}$ and satisfies the condition
\eqref{condition-equality} with $g^*=h_\infty^*$.
Therefore, $\varphi_0$ is an inflated map with respect to $(d\mu_{h_t}, h_t)$, and the positive
semidefinite metric $h_\infty^*$ is a solution to Problem \ref{problem-spec} associated with
$(d\mu_{h_t}, h_t)$.

For $t=1$, the metrics $h_t^*$, $1/\sqrt{6}\leq t\leq \infty$, all provide solutions to Problem
\ref{problem-spec} associated with $(d\mu_{h_1}, h_1)$.
\end{Proposition}

As we will see in Remark \ref{revisitBerger} below, for each $t>0$, a larger family of
left-invariant metrics on $\mathrm{SU}(2)$ gives solutions to Problem \ref{problem-spec}
associated with $(d\mu_{h_t}, h_t)$.

\subsection{Left-invariant metrics on $\mathrm{SU}(2)$}\label{ss-left-inavariant}

We consider Problems \ref{problem-emb} and \ref{problem-spec} associated with arbitrary left-invariant
metrics \eqref{left-invariant} on $\mathrm{SU}(2)$ and their common normalized volume element $d\mu$.
By change of variables and rescaling, we may assume that $a\leq b\leq c$ and $c=1$,
and set $h_{a,b} := (1/a) {\sigma_1}^2 + (1/b) {\sigma_2}^2 + {\sigma_3}^2$.
Since the metrics $h_{a,a}$ and $h_{a,1}$ are isometric to Berger metrics, possibly
up to scaling, we also assume that $a<b<1$.

We search for solutions to Problem \ref{problem-spec} associated with $(d\mu, h_{a,b})$
among left-invariant metrics.
Note that Proposition \ref{invariant-solution} justifies this strategy.
By working on the scale-invariant functional \eqref{scale-invariant-functional}, we may assume that the variable metrics
are $h^*_{u,v} := u E_1\otimes E_1 + v E_2\otimes E_2 + E_3\otimes E_3$, dual to
$h_{u,v} := (1/u) {\sigma_1}^2 + (1/v) {\sigma_2}^2 + {\sigma_3}^2$.
Therefore, the functional is
$$
\Phi(u,v) := \int_M (h_{u,v}^*, h_{a,b})\, d\mu \Bigm/ \lambda_1(d\mu,h_{u,v}^*)
= \Bigl( \frac{u}{a}+\frac{v}{b} + 1 \Bigr) \Bigm/ \lambda_1(d\mu,h_{u,v}^*).
$$
By Lemma \ref{eigensystem-su2}, we have
\begin{equation}\label{left-invariant-first-eigenvalue}
\lambda_1(d\mu, h_{u,v}^*) = \lambda_1(h_{u,v}) = \left\{
\begin{array}{cc} u+v+1, & (u,v)\in D_0,\\
4(u+v), & (u,v)\in D_1,\\
4(u+1), & (u,v)\in D_2,\\
4(v+1), & (u,v)\in D_3
\end{array} \right.
\end{equation}
with corresponding eigenfunctions
\begin{equation}\label{left-invariant-first-eigenfunction}
\zeta_1,\,\, \zeta_2;\quad
\zeta_1 \overline{\zeta_2},\,\, \zeta_1 \overline{\zeta_1} - \zeta_2 \overline{\zeta_2};\quad
\zeta_1^2 + \overline{\zeta_2}^2,\,\, \zeta_1\zeta_2 - \overline{\zeta_1}\overline{\zeta_2};\quad
\zeta_1^2 - \overline{\zeta_2}^2,\,\, \zeta_1\zeta_2 + \overline{\zeta_1}\overline{\zeta_2},
\end{equation}
respectively.
Here, $D_i$, $0\leq i\leq 3$, denote the domains in the quarter plane $Q=\{u>0, v>0\}$, given by
\begin{align*}
& D_0 = Q\setminus \bigcup_{i=1}^3 \mathrm{Int} D_i,\qquad
D_1 = Q\cap \left\{ u + v\leq \frac{1}{3} \right\},\\
& D_2 = Q\cap \left\{ - u + \frac{v}{3}\geq 1 \right\},\qquad
D_3 = Q\cap \left\{ \frac{u}{3} - v\geq 1 \right\}.  
\end{align*}
Note that $\Phi$ extends continuously to $\overline{Q}\setminus \{(0,0)\}$, since for
$(u,v)\in \partial Q\setminus \{(0,0\}$, the corresponding Bakry-\'Emery Laplacian
$-\Delta_{(d\mu, h_{u,v}^*)}$ coincides with the sub-Laplacian of the Carnot-Carath\'eodory
metric $h_{u,v}$
and \eqref{left-invariant-first-eigenvalue}, \eqref{left-invariant-first-eigenfunction} continue to hold.
Now it is an elementary matter to verify that the infimum of $\Phi$ in $Q$ is attained exactly
at $(u,v) = (0, 1/3)\in \partial D_1$ with $\Phi(0, 1/3) = \frac{(1/b)+3}{4}$.
Thus, the positive semidefinite metric $h_{0,1/3}^*$ is a candidate for a solution to
Problem \ref{problem-spec}.

We turn to finding inflated maps.
The sub-Laplacian of $h_{0, 1/3}$ has first eigenfunctions
$\zeta_1,\, \zeta_2,\, \zeta_1 \overline{\zeta_2},\, \zeta_1 \overline{\zeta_1} - \zeta_2 \overline{\zeta_2}$.
Since $a<b<1$, the map $\varphi_{p,q}$ as in Example \ref{berger-S3} with $(p,q) = (1, \sqrt{(1/b)-1})$
satisfies the constraint \eqref{short} of Problem \ref{problem-emb}.
Also, this map and the metric $h_{0, 1/3}^*$ satisfy the equality conditions for
\eqref{variance-lambda-inequality} as in Proposition \ref{variance-inequality}.
Therefore, $\varphi_{1, \sqrt{(1/b)-1}}$ and $h_{0, 1/3}^*$ give solutions to Problems \ref{problem-emb}
and \ref{problem-spec} associated with $(d\mu, h_{a,b})$, respectively.

From the above discussion, we obtain

\begin{Proposition}\label{left-inavariant-summary}
Let $h_{a,b,c}$ be the left-invariant metric on $\mathrm{SU}(2)$ given by
$$
h_{a,b,c} = \frac{1}{a} {\sigma_1}^2 + \frac{1}{b} {\sigma_2}^2 + \frac{1}{c} {\sigma_3}^2,
$$
and let $d\mu$ be the common normalized volume element of the metrics $h_{a,b,c}$.
Suppose that $a<b<c$.
Then the embedding
\begin{align*}
& \varphi\colon \mathrm{SU}(2)\to \R^7;\\
&\qquad A(\zeta_1, \zeta_2)\mapsto \Bigl( \sqrt{1/c}\, \zeta_1, \sqrt{1/c}\, \zeta_2,
\sqrt{(1/b)-(1/c)}\, \zeta_1 \overline{\zeta_2},
\frac{\sqrt{(1/b)-(1/c)}}{2}\, (\zeta_1 \overline{\zeta_1} - \zeta_2 \overline{\zeta_2})
\Bigr)
\end{align*}
is a short map by first eigenfunctions of $-\Delta_{(d\mu, h_{0,1/3,1}^*)}$ and satisfies
the condition \eqref{condition-equality} with $g^*=h_{0,1/3,1}^*$.
Therefore, $\varphi$ is an inflated map with respect to $(d\mu, h_{a,b,c})$,
and the positive semidefinite metric $h_{0, 1/3,1}^*$ gives a solution to Problem
\ref{problem-spec} associated with $(d\mu, h_{a,b,c})$.
\end{Proposition}

\begin{Remark}\label{revisitBerger}
We revisit the Berger metrics.
Note that the metric $h_{a,a,1}$ coincides with the Berger metric $h_{\sqrt{a}}$
up to scaling.
So we shall work with $h_{a,a,1}$.
It can be verified that the functional $\Phi$ achieves its infimum at each point of
$$
\left\{
\begin{array}{ll}
\partial D_1, & a<1,\\ D_0,
& a=1,\\ \partial_\infty D_0, & a>1,
\end{array}
\right.
$$
where $\partial_\infty D_0$ denotes the boundary at infinity of $D_0$.
Thus, the metrics $h_{u,v,1}^*$ with $(u,v)$ in the above range provide solutions
to Problem \ref{problem-spec} associated with $(d\mu, h_{a,a,1})$.

The case $a> 1$ requires some care.
The metric $h_{u,v,1}^*$ along $\partial D_2$ is
$
h_{u,3u+3,1}^* = u\, E_1\otimes E_1 + (3u+3) E_2\otimes E_2 + E_3\otimes E_3,
$
which coincides with $E_1\otimes E_1 + \frac{3u+3}{u} E_2\otimes E_2 + \frac{1}{u}
E_3\otimes E_3$ up to scaling.
The latter metric approaches the positive semidefinite metric
$h_{1,3,0}^* = E_1\otimes E_1 + 3\, E_2\otimes E_2$, dual to the Carnot-Carath\'eodory
metric $h_{1,3,0}$,
as $(u,v)$ tends to infinity along $\partial D_2$.
Similarly, after suitable rescaling, $h_{u,v,1}^*$ approaches
$h_{1,1/3,0}^* = E_1\otimes E_1 + \frac{1}{3} E_2\otimes E_2$
as $(u,v)$ tends to infinity along $\partial D_3$.
In view of the scale-invariance of the functional \eqref{scale-invariant-functional},
it would be appropriate to interpret the metrics corresponding to the points of
$\partial_\infty D_0$ as the positive semidefinite metrics
$h_{1,v,0}^* = E_1\otimes E_1 + v\, E_2\otimes E_2$, $1/3\leq v\leq 3$.
In fact, the functional \eqref{scale-invariant-functional} is well-defined and attains
its infimum on these metrics.
\end{Remark}

\begin{Remark}
Note that other than the standard round metric, the metric $h_{1/6,1/6,1}$
is the only left-invariant metric, up to isometry and scaling, on $\mathrm{SU}(2)$
such that Problem \ref{problem-spec} is self-solvable.
Therefore, in spite of having high multiplicity $7$, the metrics $h_{a,b,1}$ with $a+b=1/3$
do not admit an isometric minimal immersion into a sphere by first eigenfunctions
unless $a=b=1/6$.
In fact, we can directly verify that these metrics are not extremal unless $a=b=1/6$
by computing the Berger functional on left-invariant metrics.
The computation also shows that the metric $h_{1/6,1/6,1}$ is a saddle-like
extremal point of the Berger functional.
\end{Remark}

\section{Nadirashvili-type theorem}

In this section, we shall prove an analogue of the Nadirashvili minimal surface theorem.
Our Nadirashvili-type theorem asserts that if a solution to Problem \ref{problem-spec} exists
and is positive definite everywhere, then a map consisting of first eigenfuctions of the
corresponding Bakry-\'Emery Laplacian gives a solution to Problem \ref{problem-emb}.
Throughout this section, let $M$ be a compact manifold of dimension $n$ equipped with a smooth
volume element $d\mu$ and a Riemannian metric $h$.

As a preliminary for the proof of the Nadirashvili-type theorem, we compute the variation of an
eigenvalue of the Bakry-\'Emery Laplacian $-\Delta_{(d\mu,g^*)}$ when the metric $g$ is smoothly deformed
(while the volume element $d\mu$ is fixed).

\begin{Lemma}\label{ev-variation}
Let $g_t^*$, $t\in (-\varepsilon, \varepsilon)$, be a smooth family of positive definite metrics
on $T^*M$.
Set $g^* := g_0^*$ and $\dot{g}^* := \frac{d}{dt} g_t^*|_{t=0}$.
Suppose that there exist a smooth function $\lambda(t)$ and a smooth family of $C^\infty$-functions
$\{ u(t) \}_{t\in (-\varepsilon, \varepsilon)}\subset C^\infty(M)$ which satisfy
$-\Delta_{(d\mu,g_t^*)} u(t) = \lambda(t) u(t)$ and $\int_M {u(t)}^2\, d\mu=1$.
Then
$$
\frac{d}{dt} \lambda(t)|_{t=0} = \int_M (\dot{g}^*, du\otimes du)\, d\mu,
$$
where $u = u(0)$.
\end{Lemma}

\begin{proof}
Note that
$
\lambda(t) = \int_M |du(t)|_{g_t^*}^2\, d\mu.
$
We compute
\begin{eqnarray*}
\frac{\partial}{\partial t} |du(t)|_{g_t^*}^2\big|_{t=0} &=& \frac{\partial}{\partial t} \left(
{g_t}^{ij} \frac{\partial u(t)}{\partial x^i} \frac{\partial u(t)}{\partial x^j} \right) \big|_{t=0}\,\,
=\dot{g}^{ij} \frac{\partial u}{\partial x^i} \frac{\partial u}{\partial x^j}
+ 2 g^{ij} \frac{\partial u}{\partial x^i} \frac{\partial \dot{u}}{\partial x^j}\\
&=& (\dot{g}^*, du\otimes du) + 2 (du,d\dot{u})_{g^*},
\end{eqnarray*}
where $\dot{u}=\frac{d}{dt}u(t)|_{t=0}$.
Integration by parts gives
$$
\int_M (du,d\dot{u})_{g^*}\, d\mu = - \int_M (\Delta_{(d\mu,g^*)} u) \dot{u}\, d\mu = \lambda(0) \int_M u\dot{u}\, d\mu
= 0,
$$
since $\int_M {u(t)}^2\, d\mu=1$.
Therefore, we obtain the desired formula.
\end{proof}

In order to apply Lemma \ref{ev-variation}, we need the following lemma on analyticity of eigenvalues
of the Bakry-\'Emery Laplacian with respect to an analytic perturbation of positive defnite metric on $T^*M$.
The lemma is an analogue of the result of Berger \cite{Berger} and Bando-Urakawa \cite{BandoUrakawa}
for the Riemannian Laplacian.

\begin{Lemma}\label{analyticity}
Let $g_t^*$, $t\in (-\varepsilon, \varepsilon)$, be an analytic family of positive-definite metrics
on $T^*M$ of the form $g_t^* = \gamma(t) ( g^*+tk^*)$, where $k^*\in \mathcal{S^*}$ and $\gamma(t)$ is
a real analytic function such that $\gamma(0)=1$.
Let $\lambda$ be an eigenvalue of $-\Delta_{(d\mu,g^*)}$ of multiplicity $m$.
Then there exist $0<\varepsilon'<\varepsilon$, real analytic functions $\lambda^{(i)}(t)$ of
$t\in (-\varepsilon', \varepsilon')$ and analytic families of $C^\infty$ functions
$\{ u^{(i)}(t) \}_{t\in (-\varepsilon', \varepsilon')}\subset C^\infty(M)$, $1\leq i\leq m$, such that
\begin{enumerate}
\renewcommand{\theenumi}{\roman{enumi}}
\renewcommand{\labelenumi}{(\theenumi)}
\item $\lambda^{(i)}(0) = \lambda$, $1\leq i\leq m$,
\item $-\Delta_{(d\mu, g_t^*)} u^{(i)}(t) = \lambda^{(i)}(t) u^{(i)}(t)$, $1\leq i\leq m$,
$t\in (-\varepsilon', \varepsilon')$,
\item $\{ u^{(i)}(t) \}_{1\leq i\leq m}$ is orthonormal in $L^2(d\mu)$,
$t\in (-\varepsilon', \varepsilon')$.
\end{enumerate}
\end{Lemma}

\begin{proof}
One may verify that $\{-\Delta_{(d\mu, g_t^*)}\}_{t\in (-\varepsilon, \varepsilon)}$ is an analytic family
of unbounded self-adjoint operators in the Hilbert space $L^2(M, d\mu)$ with common domain of definition
and with compact resolvent.
Therefore, by the result due to Rellich \cite{Rellich1, Rellich2} (see also \cite{Kato, KrieglMichor}),
the eigenvalues and the eigenvectors of $-\Delta_{(d\mu, g_t^*)}$ may be parametrized real analytically
in $t$, as asserted in the lemma.
\end{proof}

\begin{Theorem}\label{nadirashvili-type-thm}
Suppose that a positive definite metric $g^*$ on $T^*M$ is a solution to Problem \ref{problem-spec}
associated with $(d\mu, h)$.
Then there exist first eigenfunctions $u_1,\dots,u_N$ of $-\Delta_{(d\mu,g^*)}$ such that the map
$\varphi = (u_1,\dots,u_N)\colon M\to \R^N$ is an isometric immersion with respect to the metric $h$.
Therefore, $\varphi$ is an inflated map with respect to $(d\mu, h)$.   
\end{Theorem}

\begin{proof}
The proof is similar to that of the Nadirashvili minimal surface theorem due to El Soufi-Ilias
\cite{ElSoufiIlias}.
By rescaling, we may assume that the metric $g^*$ satisfies $\int_M ( g^*, h )\, d\mu = 1$.
For any $k^*\in \mathcal{S}^*$ satisfying $\int_M ( k^*, h )\, d\mu = 0$,
there exists an analytic family of metrics $g_t^*$, $t\in (-\varepsilon, \varepsilon)$, on $T^*M$
such that $g_0^*=g^*$, $\dot{g}^*=k^*$ and $\int_M ( g_t^*, h)\, d\mu = 1$ for all $t$.
In fact, it suffices to set $g_t^* := g^*+tk^* \bigm/ \int_M ( g^*+tk^*, h)\, d\mu$.
Let $m$ denote the multiplicity of $\lambda_1(d\mu, g^*)$.
Let $\lambda^{(i)}(t)$ and $\{ u^{(i)}(t) \}$ be real analytic functions and analytic families
of $C^\infty$-functions as in Lemma \ref{analyticity}.
Then by Lemma \ref{ev-variation}, we have
$
\frac{d}{dt} \lambda^{(i)}(t)|_{t=0} =  \int_M ( k^*, du^{(i)}(0)\otimes du^{(i)}(0) )\, d\mu.
$

Now the function $t\mapsto \lambda_1(d\mu, g_t^*)$ has right and left derivatives:
$$
\frac{d}{dt} \lambda_1(d\mu, g_t^*)|_{t=0^+} = \min_{1\leq i\leq m} \frac{d}{dt} \lambda^{(i)}(t)|_{t=0},
\quad
\frac{d}{dt} \lambda_1(d\mu, g_t^*)|_{t=0^-} = \max_{1\leq i\leq m} \frac{d}{dt} \lambda^{(i)}(t)|_{t=0}.
$$
Since $\lambda_1(d\mu, g_t^*)$ takes maximum at $t=0$, we have
$
\frac{d}{dt} \lambda_1(d\mu, g_t^*)|_{t=0^+}\leq 0\leq \frac{d}{dt} \lambda_1(d\mu, g_t^*)|_{t=0^-}.
$
Therefore, by the intermediate value theorem, we conclude that there exists
$u\in E_1(d\mu,g^*)\setminus \{0\}$ such that $\int_M ( k^*, du\otimes du )\, d\mu = 0$,  
where $E_1(d\mu,g^*)$ denotes the first eigenspace of $-\Delta_{(d\mu, g^*)}$.
We have shown  

\smallskip\noindent
{\em Claim 1.}\quad For any $k^*\in \mathcal{S}^*$ satisfying $\int_M ( k^*, h )\, d\mu = 0$,
there exists $u\in E_1(d\mu,g^*)\setminus \{0\}$ such that $\int_M ( k^*, du\otimes du )\, d\mu = 0$.

\smallskip
Let $\mathcal{C}$ denote the convex hull of the set $\{ du\otimes du \mid u\in E_1(d\mu,g^*) \}$
in $\mathcal{S}$.
This is a closed convex cone in $\mathcal{S}$.
Then we have

\smallskip\noindent
{\em Claim 2.}\quad $h\in \mathcal{C}$.

\smallskip\noindent
Suppose that $h\notin \mathcal{C}$.
Then by the hyperplane separation theorem, there exists $k^*\in \mathcal{S}^*$ such that
$\langle k^*, l \rangle \geq 0$ for all $l\in \mathcal{C}$
and $\langle k^*, h \rangle < 0$.
Here, $\langle k^*, l \rangle := \int_M ( k^*, l )\, d\mu$ for $k^*\in \mathcal{S}^*$
and $l\in \mathcal{S}$.
Set
$
{k^*}' = k^* - \frac{\langle k^*, h \rangle}{\langle h^*, h \rangle} h^*,
$
so that $\langle {k^*}', h \rangle = 0$. Then for any $u\in E_1(d\mu,g^*)\setminus \{0\}$,
$$
\langle {k^*}', du\otimes  du \rangle = \langle k^*, du\otimes du \rangle
- \frac{\langle k^*, h \rangle}{\langle h^*, h \rangle} \langle h^*, du\otimes du \rangle > 0.
$$
This contradicts Claim 1.

Claim 2 implies that there exist $N\in \N$ and $u_k\in E_1(d\mu,g^*)$, $1\leq k\leq N$, such that
$h = \sum_{k=1}^N du_k\otimes du_k$.
Thus, the map $\varphi := (u_1,\dots, u_N)\colon M\to \R^N$ satisfies $h = \varphi^* h_{\R^N}$.
\end{proof}

\begin{Remark}\label{remark-to-n-type-thm}
(i)\,\,
The theorem implies that the multiplicity of the first eigenvalue $\lambda_1(d\mu, g^*)$ of
$-\Delta_{(d\mu,g^*)}$ is greater than or equal to $n+1$.

\smallskip\noindent
(ii)\,\, Under the assumption of the theorem, it follows from the last statement of Proposition
\ref{variance-inequality} that any inflated map with respect to $(d\mu, h)$ arises as in the theorem
and is an isometric immersion with respect to $h$.
\end{Remark}

By Theorem \ref{nadirashvili-type-thm}, we can characterize the metrics $h$ for which
Problem \ref{problem-spec} associated with $(d\mu_h, h)$ is self-solvable
(cf.~ Proposition \ref{minimal-to-sphere}).

\begin{Corollary}\label{self-solvable}
Let $(M,h)$ be a compact Riemannian manifold.
Then $g^*=h^*$ gives a solution to Problem \ref{problem-spec} associated with $(d\mu_h, h)$
if and only if $(M,h)$ admits an isometric minimal immersion into a sphere (of some radius)
by first eigenfunctions of $-\Delta_h$.
\end{Corollary}

\section{Inflation dimension}

G\"oring-Helmberg-Wappler \cite{GoeringHelmbergWappler2} introduced the notion of
{\em rotational dimension} for a finite graph in terms of solutions to their
graph-embedding problem.
We can define an analogue of it for a manifold.

For a map $\varphi$ from a manifold into $l^2$, the (possibly infinite)
dimension of the minimal closed affine subspace of $l^2$ containing the image of $\varphi$
is called the {\em dimension} of $\varphi$.

\begin{Definition}
Let $M$ be a compact manifold.
If a smooth volume element $d\mu$ and a Riemannian metric $h$ are given, the minimum
dimension of a solution to Problem \ref{problem-emb} is denoted by $\infldim (d\mu,h)$
and is called the {\em inflation dimension} of $M$ with respect to $(d\mu, h)$.
If no solutions exist, set $\infldim (d\mu,h)=\infty$.

The supremum of $\infldim (d\mu,h)$ over all choices of $(d\mu, h)$ is denoted by
$\infldim (M)$ and is called the inflation dimension of $M$.
\end{Definition}

The inflated maps found in Section \ref{examples} provide upper bounds for the invariant
$\infldim (d\mu,h)$ when $d\mu=d\mu_h$.

\begin{Example} (i)\,\,
Let $h$ be a flat metric on $T^2$. Then
$$
\infldim (d\mu_h,h)\leq \left\{ \begin{array}{cl} 4, & \mbox{if $h$ is rectangular},\\
6, & \mbox{otherwise}. \end{array} \right.
$$

\smallskip\noindent
(ii)\,\, Let $h_t$, $t>0$, be a Berger metric on $S^3$. Then
$$
\infldim (d\mu_{h_t},h_t)\leq \left\{ \begin{array}{cl} 7, & \mbox{if $t<1$},\\
4, & \mbox{if $t\geq 1$}. \end{array} \right.
$$

\smallskip\noindent
(iii)\,\, Let $h_{a,b,c}$ be the left-invariant metric on $\mathrm{SU}(2)$ as in
\eqref{left-invariant} with $a<b<c$.
Then $\infldim (d\mu_{h_{a,b,c}},h_{a,b,c})\leq 7$.
\end{Example}

We will discuss the inflation dimension of a manifold closely in a forthcoming work.

\end{document}